\documentclass[12pt]{amsart}
\usepackage{amssymb}
\usepackage{enumerate}

        \usepackage{graphicx}
        \usepackage[usenames,dvipsnames,svgnames,table]{xcolor}
        \usepackage[a4paper]{geometry}

\newcommand{\N}{{\mathbb N}}
\newcommand{\C}{{\mathbb C}}

\newcommand{\Z}{{\mathbb Z}}

\newcommand{\wand}{wandering domain}

\newcommand{\tef}{transcendental entire function}

\newcommand{\nhd}{neighbourhood}
\newcommand{\sconn}{simply connected}
\newcommand{\mconn}{multiply connected}

\newcommand{\sw}{spider's web}

\newcommand{\spl}{strongly polynomial-like}
\newcommand\qfor{\quad\text{for }}

\newcommand{\kfc}{I^{\!+\!}(f)}
\newcommand{\kfnc}{I^{\!+\!}(f^n)}

\theoremstyle{plain}
\newtheorem{theorem}{Theorem}[section]
\newtheorem{corollary}[theorem]{Corollary}

\newtheorem*{theorem*}{Theorem}
\newtheorem*{proposition*}{Proposition}

\newtheorem{lemma}[theorem]{Lemma}
\theoremstyle{definition}

\theoremstyle{remark}
\newtheorem*{remark*}{Remark}
\newtheorem*{remarks*}{Remarks}

\theoremstyle{problem}

\theoremstyle{example}
\newtheorem{example}[theorem]{Example}
\newtheorem*{example*}{Example}
\theoremstyle{question}

\theoremstyle{questions}
\newtheorem*{questions*}{Questions}

{\begin{list}{}%
         {\setlength{\leftmargin}{#1}}%
         \item[]%
}
{\end{list}}

\begin{document}


\title[Connectedness properties of $ \kfc $]{Connectedness properties of the set where the iterates of an entire function are unbounded}

\author{J.W. Osborne, P.J. Rippon \and G.M. Stallard}
\address{Department of Mathematics and Statistics \\
   The Open University \\
   Walton Hall\\
   Milton Keynes MK7 6AA\\
   UK}
\email{john.osborne@open.ac.uk, phil.rippon@open.ac.uk, gwyneth.stallard@open.ac.uk}



\thanks{The last two authors are supported by the EPSRC grant EP/K031163/1.}


\begin{abstract}
We investigate the connectedness properties of the set $ \kfc $ of points where the iterates of an entire function $ f $ are unbounded.  In particular, we show that $ \kfc $ is connected  whenever  iterates of the minimum modulus of $ f $ tend to $ \infty $. For a general \tef\  $ f $,  we show that $ \kfc \cup \lbrace \infty \rbrace$ is always connected and that, if $ \kfc $ is disconnected, then it has uncountably many components, infinitely many of which are unbounded.
\end{abstract}
\maketitle

\section{Introduction}
\label{intro}
\setcounter{equation}{0}

Denote the $ n $th iterate of an entire function $ f $ by $ f^n $, for $ n \in \N. $  The \textit{Fatou set} $ F(f) $ is the set of points $ z \in \C $ such that the family of functions $ \lbrace f^n : n \in \N \rbrace $ is normal in some \nhd\ of $ z $, and the \textit{Julia set} $ J(f) $ is the complement of $ F(f) $.  We refer to \cite{aB, wB93, Mil}, for example, for an introduction to complex dynamics and the properties of these sets.

For any $ z \in \C $, we call the sequence $ (f^n(z))_{n \in \N} $ the \textit{orbit} of $ z $ under $ f $.  This paper is concerned with the set of points whose orbits are unbounded, which we denote by
\[ \kfc = \lbrace z \in \C : (f^n(z))_{n \in \N} \textrm{ is unbounded} \rbrace. \]
Clearly, $ \kfc $ contains the escaping set,
\[ I(f) = \lbrace z \in \C : f^n(z) \to \infty \text{ as } n \to \infty \rbrace, \]
and is the complement of $ K(f) $, the set of points whose orbits are bounded.  If $ f $ is a polynomial, then $ K(f) $ is the \textit{filled Julia set} of $ f $, and it is well known that $ \kfc = I(f) $.  However, if $ f $ is transcendental, then $ \kfc  \setminus I(f) $ always meets $ J(f) $ and may also meet $ F(f) $; see \cite{OS} and references therein for the properties of $ \kfc  \setminus I(f) $.

For a general \tef, we show in Section \ref{basic} that $\kfc$ has many properties in common with $I(f)$. For example, we show that the properties of $  I(f) $ proved by Eremenko in \cite{E} also hold for $ \kfc $, and we prove the following result, which parallels \cite[Theorem~4.1]{RS11}.

\begin{theorem}
\label{infty}
Let $ f $ be a \tef.  Then $ \kfc \cup \lbrace \infty \rbrace $ is connected.
\end{theorem}

In the paper \cite{E}, Eremenko remarked that it is plausible that $ I(f) $ has no bounded components.  This conjecture has stimulated much research in transcendental dynamics and remains open, though there have been several partial results \---\ see for example \cite{R07, RS05, R3S}.  One of the strongest partial results for a general \tef\ \cite[Theorem 1]{RS05} was obtained by considering the fast escaping set $ A(f) $, a subset of $ I(f) $ defined in terms of the iterated maximum modulus function. By contrast, in this paper we show that the connectedness properties of the \textit{superset} $ \kfc $ of $ I(f) $ are related to a  completely new condition involving  the iterated  \textit{minimum} modulus function.

We prove the following result,  in which
\begin{equation*}
m(r) = m(r,f) := \min \lbrace \vert f(z) \vert : \vert z \vert = r \rbrace,
\end{equation*}
and $ m^n(r) $ denotes the $ n $th iterate of the function $ r \mapsto m(r) $.

\begin{theorem}
\label{minconn}
Let $ f $ be a \tef\ for which
\begin{equation}\label{minmodprop}
\text{there exists } r > 0 \text{ such that } m^n(r) \to \infty \text{ as } n \to \infty.
\end{equation}
Then $ \kfc $ is connected.
\end{theorem}

In fact, we show that $ \kfc $ is connected for a more general class of functions than is covered by Theorem \ref{minconn}. Details are given in Section \ref{connected}.

It is natural to ask which \tef s satisfy the condition~\eqref{minmodprop}. Clearly, there are many that do not \---\ for example, any function bounded on a path to~$\infty$. However, there are also functions that do satisfy the condition.  In forthcoming work, we consider the consequences of condition~\eqref{minmodprop} for other sets related to $I(f)$ and $\kfc$, and show that there are many classes of functions for which  condition~\eqref{minmodprop}  holds.  In particular, we show that this is the case for all entire functions of order less than~1/2, so $ \kfc $ is connected for such functions. It is an interesting question whether condition~\eqref{minmodprop} is also sufficient to ensure that $I(f)$ is connected.

Note that there are \tef s for which $ \kfc $ is disconnected.  For example,  if $ f(z) = \sin z $, then $ f $ maps the real line $ \mathbb{R} $ onto the interval $ [ -1, 1], $ so $ \mathbb{R} $ is a closed, connected set in $ K(f) $ that disconnects $ \kfc $.

In Section \ref{components}, we prove a number of results on the components of $ \kfc $ for a general \tef, including the following.

\begin{theorem}
\label{uncount}
Let $ f $ be a \tef\ such that $ \kfc $ is disconnected.  Then $ \kfc $ has uncountably many components, infinitely many of which are unbounded.
\end{theorem}

The paper is organised as follows.  In Section \ref{basic}, we prove Theorem~\ref{infty} and some basic properties of $\kfc$, and in Section \ref{connected} we give the proof of Theorem~\ref{minconn} and related results.  In Section \ref{components}, we prove Theorem~\ref{uncount} and a number of other results on the components of $ \kfc $. Finally, in Section \ref{examples}, we give some examples related to the hypotheses of Theorem~\ref{minconn} and its generalisation in Section~\ref{connected}.

\section{Basic properties of $\kfc$}
\label{basic}
\setcounter{equation}{0}

In this section we   prove   a number of basic properties of $\kfc$ and discuss the interaction of $\kfc$ with the Fatou set and the Julia set.  We note first that, for a \tef\ $ f $, it follows immediately from the corresponding properties of $ K(f) $ that $ \kfc $ is completely invariant and that $ \kfnc = \kfc $, for $ n \in \N. $

As usual, we refer to components of the Fatou set as \emph{Fatou components}.  If $ U $ is a Fatou component of $ f $ then, for every   $ n \in \N $,   $ f^n(U) \subset U_n $ for some Fatou component~$ U_n $. A simple normality argument shows that, if $ U \cap \kfc \neq \emptyset $, then $ U \subset \kfc. $

The   following properties of $I(f)$ were proved   by Eremenko~\cite{E}:
\begin{equation*}\label{Iprops}
I(f)\neq \emptyset,\;\; I(f)\cap J(f)\neq\emptyset,\;\;J(f)=\partial I(f),
\end{equation*}
and $\overline{I(f)}$ has no bounded components.

The proofs that these properties   also   hold for $\kfc$ are similar to those for $I(f)$ so we give only brief details.

\begin{theorem}\label{Qprops}
Let $f$ be a {\tef}. Then
\begin{equation}\label{Qprops1}
\kfc\neq \emptyset,\;\; \kfc\cap J(f)\neq\emptyset,\;\;J(f)=\overline{\kfc\cap J(f)},\;\;J(f)=\partial \kfc,
\end{equation}
and $\overline{\kfc}$ has no bounded components.
\end{theorem}

\begin{remark*}
In view of the considerable interest in Eremenko's conjecture, mentioned in Section~\ref{intro}, it is natural to ask whether all the components of $\kfc$ are unbounded.
\end{remark*}

\begin{proof}[Proof of Theorem~\ref{Qprops}]
The first two properties in \eqref{Qprops1}   follow immediately from the corresponding properties of $ I(f) $, and the fact that $ I(f) \subset \kfc. $  The   third property follows from the second by the blowing up property of $J(f)$; see Lemma~\ref{blowup}.

 It follows from the third property in \eqref{Qprops1} that $J(f)\subset \overline{\kfc}$.  Now since the repelling periodic points of $ f $ are dense in $ J(f) $ (see \cite[Theorem 1]{Bak68}), any open set $ G \subset \kfc $ satisfies $ G \subset F(f) $.  Hence $J(f)\subset \partial \kfc$. On the other hand, no point of $\partial \kfc$ can lie in $F(f)$, since any such point would have a neighbourhood in $\kfc$. We conclude that $J(f)=\partial \kfc$.

Finally, if $\overline{\kfc}$ has a bounded component, $E$ say, then there is an open topological annulus $A$    that surrounds $E$ and lies in the complement of $\overline{\kfc}$.   Since $\overline{\kfc}$ is completely invariant under $f$, we deduce by Montel's theorem that   $A$ lies in a component of $ F(f)$, and this component must be \mconn\ since $J(f)=\partial \kfc$.    But any {\mconn} Fatou component of~$f$ is contained in $I(f)$ (see~\cite[Theorem~3.1]{iB84}) and hence in $\kfc$, so we obtain a contradiction. This completes the proof of Theorem~\ref{Qprops}.
\end{proof}

\begin{corollary}\label{openclosed}
Let~$f$ be a \tef. Then $\kfc$ is neither open nor closed.
\end{corollary}
\begin{proof}
  If $\kfc$ is open, then this implies that $\kfc\subset F(f)$, which is a contradiction since $\kfc\cap J(f)\ne\emptyset$.  If $ \kfc $ is closed, then since  $J(f)=\partial \kfc$ we have $ J(f) \subset \kfc $, which is again a contradiction.
\end{proof}

Next, we consider Fatou components in $ \kfc $ and their boundaries. A Fatou component $ U $ in $ \kfc $  must  be a Baker domain, a preimage of a Baker domain or a wandering domain.  The definitions are as follows, where as before $ U_n $ denotes the Fatou component containing $ f^n(U) $, for $ n \in \N $:
\begin{itemize}
\item if $ U = U_p $ for some $ p \in \N $,  so $ U $  is periodic with period $ p $, then $ U $ is a \textit{Baker domain} and has the property that $ f^{np}(z) \to \infty $ as $ n \to \infty $ for all $ z \in U; $
\item if $ U $ is not eventually periodic,  that is, $ U_m \neq U_n $ whenever $ m \neq n $,  then $ U $ is a \textit{wandering domain}.
\end{itemize}

We refer to \cite{wB93}, for example, for further information on the classification of Fatou components.

We now state a number of results on the boundaries of the possible types of Fatou components in $ \kfc $, and prove a simple consequence of these results that we use later in the paper.

Clearly, Baker domains and their preimages lie in $ I(f) $.  A Baker domain $ U $ of period $ p $ is said to be \emph{univalent} if $ f^p $ is univalent in $ U. $  Our first lemma is a simple corollary of \cite[Theorem 1.1]{RS14}.

\begin{lemma}
\label{unibak}
Let $ f $ be a \tef\ and let $ U $ be a univalent Baker domain of $ f $.  Then $ \partial U \cap I(f) \neq \emptyset. $
\end{lemma}

It is an interesting open question whether the conclusion of Lemma \ref{unibak} applies for \emph{any} Baker domain $U$ of a \tef. The following property of the boundary of a non-univalent Baker domain was proved by Baker and Dom\'{i}nguez \cite[Corollary 1.3]{BD1}.

\begin{lemma}
\label{nonunibak}
Let $ f $ be a \tef\ and let $ U $ be a Baker domain of $ f $ such that $ f $ is not univalent in $ U $.  Then $ \partial U $ is disconnected.
\end{lemma}

The next lemma follows from a general result on the boundaries of wandering domains \cite[Theorem 1.5]{OS}.

\begin{lemma}
\label{kfcwand}
Let $ f $ be a \tef\ and let $ U $ be a wandering domain of $ f $ such that $ U \subset \kfc $.  Then $ \partial U \cap \kfc \neq \emptyset. $
\end{lemma}

We now prove the following consequence of Lemmas \ref{unibak}, \ref{nonunibak} and \ref{kfcwand}.

\begin{lemma}
\label{kfcjf}
Let $ f $ be a \tef.  Then every component of $ \kfc $ that is neither a non-univalent Baker domain nor a preimage of such a domain must meet $ J(f) $.  In particular, every component of $ \kfc $ with connected boundary meets $ J(f). $
\end{lemma}

\begin{proof}
Suppose that $ f $ is a \tef, and that some component $ C $ of $ \kfc $ does not meet $ J(f) $ and is neither a non-univalent Baker domain nor a preimage of such a domain.  Then $ C \subset U $ for some Fatou component $ U \subset \kfc $, and indeed $ C = U $ since $ C $ is a component of $ \kfc $.  Since $ C $ is not a non-univalent Baker domain or a preimage of such a domain, either some iterate of $ f $ maps $ C $ to a univalent Baker domain, or $ C $ is wandering domain.

Using the fact that~$f$ maps any boundary point of a Fatou component to a boundary point of a Fatou component, we deduce that
\begin{itemize}
\item if $ C $ is mapped to a univalent Baker domain, then $ \partial C \cap \kfc \neq \emptyset $ by Lemma \ref{unibak}, and
\item if $ C $ is a \wand, then $ \partial C \cap \kfc \neq \emptyset $ by Lemma \ref{kfcwand}.
\end{itemize}
In either case we have a contradiction, since $ \partial C \subset J(f) $ and $ C \cup \{\zeta\} $ is connected for any $\zeta\in \partial C$.

The final statement of the lemma follows from the fact that non-univalent Baker domains and their preimages have disconnected boundaries, by Lemma~\ref{nonunibak}.
\end{proof}

\begin{remark*}
If we were able to show that $ \partial U \cap I(f) \neq \emptyset $ for \emph{any} Baker domain $ U $ of a \tef\ $ f $, then it would follow that every component of $ \kfc $ meets $ J(f) $.
\end{remark*}

We now give the proof of Theorem~\ref{infty}. In fact, we prove the following, which includes a useful equivalent result.

\begin{theorem}
\label{bounddomkfc}
If $ f $ is a \tef, then the following statements hold and are equivalent.
\begin{enumerate}[(a)]
\item If $ G $ is a bounded, \sconn\ domain such that $ G \cap \kfc \neq \emptyset $, then $ \partial G \cap \kfc \neq \emptyset. $
\item $ \kfc \cup \{ \infty \} $ is connected.
\end{enumerate}
\end{theorem}

Note that if $E\cup \{\infty\}$ is connected, where $E$ is a subset of $\C$, then it does not follow that the components of~$E$ are all unbounded, unless~$E$ is closed.

The proof of Theorem \ref{bounddomkfc} is similar to that of the corresponding result for $I(f)$ given in~\cite[Theorem 4.1]{RS11}.  In particular, we use the following lemma.

\begin{lemma} \cite[Lemma 4.1]{RS11}
\label{bounddomif}
Let $ f $ be a \tef.  If $ G $ is a bounded, \sconn\ domain such that $ G \cap J(f) \neq \emptyset $, then $ \partial G \cap I(f) \neq \emptyset. $
\end{lemma}

\begin{proof}[Proof of Theorem \ref{bounddomkfc}]
We first prove~(a), and then show that this implies~(b).  Since it is clear that~(b) implies~(a), this will prove the theorem.

Let $ G $ be a bounded, \sconn\ domain that meets $ \kfc. $ If $ G \cap J(f) \neq \emptyset, $  then $ \partial G \cap \kfc \neq \emptyset $ by Lemma \ref{bounddomif}.  Thus we may assume that $ G \subset U $ for some Fatou component $ U \subset \kfc, $ and indeed that $ G = U $, because otherwise we again have $ \partial G \cap \kfc \neq \emptyset $.  Since Baker domains and their preimages are unbounded, $ G $ must be a wandering domain, and it follows from Lemma \ref{kfcwand} that we again have $ \partial G \cap \kfc \neq \emptyset. $  This proves statement~(a).

To show that (a) implies (b), suppose that $ \kfc \cup \{ \infty \} $ is not connected.  Then there exist disjoint open sets $ G_1, G_2 \subset \widehat{\C} $ such that
\begin{equation}
\label{connect}
\kfc \cup \{ \infty \} \subset G_1 \cup G_2
\end{equation}
and
$$ G_i \cap (\kfc \cup \{ \infty \}) \neq \emptyset, \qfor i = 1,2.$$
We can assume that $ G_1 $ is bounded and \sconn, and that $ \infty \in G_2. $  Since $ G_1 $ meets $ \kfc $, it follows from (a) that $ \partial G_1 \cap \kfc \neq \emptyset, $ which contradicts~(\ref{connect}).  Thus $ \kfc \cup \{ \infty \} $ is connected, as required.
\end{proof}

\section{Proof of Theorem \ref{minconn}}
\label{connected}
\setcounter{equation}{0}
Theorem \ref{minconn} states that, if $ f $ is a \tef\ and there exists $ r>0 $ such that $ m^n(r) \to \infty, $ then $ \kfc $ is connected.  In this section, we prove the following more general result, which shows that $ \kfc $ is connected for an even wider class of functions.

We use the notation $ \N_0 $ for the set of non-negative integers $ \N \cup \lbrace 0 \rbrace $, and we say that a set $ A \subset \C $ \textit{surrounds} a set $ B \subset \C $ if $ B $ lies in a bounded component of the complement of $ A. $

\begin{theorem}
\label{fullresult}
Suppose that $ f $ is a \tef\ and there exists a sequence of bounded, \sconn\ domains  $ (D_n)_{n \in \N_0} $ such that
\begin{enumerate}[(a)]
\item $ f(\partial D_n) \text{ surrounds } D_{n+1}, \text{ for } n \in \N_0, $ and
\item every disc centred at $ 0 $ is contained in $ D_n $ for sufficiently large $ n $.
\end{enumerate}
Then $ \kfc $ is connected.
\end{theorem}

Before proving this result we make some remarks.

\begin{enumerate}[(1)]
\item Suppose that, for a \tef\ $ f $, there exists $ r>0 $ such that $ m^n(r) \to \infty $ as $ n \to \infty $.  Define
\[ D'_n = \lbrace z \in \C : \vert z \vert < m^n(r) \rbrace, \, \textrm{ for } n \in \N_0.\]
Since  every \tef\ has points of period $ 2 $,  it follows from the definition of the minimum modulus function that, for some $ N \in \N_0, $
\[ f(\partial D'_n) \textrm{ surrounds } D'_{n+1}, \textrm{ for } n \geq N. \]
Moreover, since $ m^n(r) \to \infty $ as $ n \to \infty $, every disc centred at $ 0 $ is contained in $ D'_n $ for sufficiently large $ n $.  If we now put
\[ D_{n} = D'_{n+N},\quad \textrm{ for } n \geq 0, \]
then it is evident that the conditions of Theorem \ref{fullresult} are satisfied for the sequence of domains $ (D_n)_{n \in \N_0} $.  Thus Theorem~\ref{minconn} follows from Theorem~\ref{fullresult}. \\
There are, however, \tef s that meet the conditions of Theorem \ref{fullresult} but not those of Theorem \ref{minconn}; see Example~5.1.\\

\item Many of the functions that meet the conditions of Theorem \ref{fullresult} are \textit{\spl}, in the sense defined in \cite{O12a}, and so they have the nice properties of such functions proved in that paper. Strongly polynomial-like functions can be characterised \cite[Theorem~1.6]{O12a} as those \tef s for which there exists a sequence of bounded, \sconn\ domains $ (D'_n)_{n \in \N_0} $ such that\\
\begin{enumerate}[(i)]
\item $ f(\partial D'_n) $ surrounds  $ \overline{D'_n} $, for $ n \in \N_0, $
\item $ \bigcup_{n \in \N_0} D'_n = \C $, and
\item $ \overline{D'_n} \subset D'_{n+1} $, for $ n \in \N_0 $.\\
\end{enumerate}
It is easy to see that~$f$ is \spl\ if it satisfies the conditions of Theorem~\ref{minconn}, and also if it satisfies the conditions of Theorem~\ref{fullresult} together with a condition such as `$\overline{D}_n\subset D_{n+1}$, for arbitrarily large values of~$n$'.

Note that not all \spl\ functions meet the conditions of Theorem~\ref{fullresult} -- indeed there are \spl\ functions for which $ \kfc $ is disconnected; see Example~5.2.
\end{enumerate}

The following lemma contains the  key induction  step in the proof of Theorem~\ref{fullresult}.
\newpage
\begin{lemma}
\label{genctm}
Suppose that $ f $ is a \tef\ and there exists a sequence of bounded, \sconn\ domains  $ (D_n)_{n \in \N_0} $ such that
\begin{enumerate}[(a)]
\item $ f(\partial D_n) \text{ surrounds } D_{n+1}, \text{ for } n \in \N_0, $ and
\item every disc centred at $ 0 $ is contained in $ D_n $ for sufficiently large $ n $.
\end{enumerate}

Suppose that, for some $ j \in \N_0 $, there exists $n_j\in \N_0$ and a continuum $ \Gamma_{n_j} $ with the following properties:
\begin{enumerate}[(i)]
\item $ \Gamma_{n_j} \subset K(f) \cap (\C\setminus D_{n_j})$;
\item there is a point $ z_{n_j} \in \Gamma_{n_j} \cap \partial D_{n_j}$;
\item there is a point $ z'_{n_j} \in \Gamma_{n_j} $ such that $ f^n(z'_{n_j}) \in D_{n_j+n} $ for all $ n \in \N$.

\end{enumerate}
Then there exists $\, n_{j+1}>n_j $ and a continuum $ \Gamma_{n_{j+1}} \subset f^{n_{j+1} - n_j}(\Gamma_{n_j}) $ such that properties~(i),~(ii) and~(iii) hold with $n_j$ replaced by $n_{j+1}$ throughout.
\end{lemma}

The proof of Lemma~\ref{genctm} depends on the following result from plane topology; see \cite[page~84]{New}.

\begin{lemma}\label{Newman}
If $E_0$ is a continuum in $\hat{\C}$, $E_1$ is a closed subset of $E_0$ and $C$ is a component of $E_0\setminus E_1$, then $\overline{C}$ meets $E_1$.
\end{lemma}

\begin{proof}[Proof of Lemma~\ref{genctm}]
Since $z_{n_j}\in \Gamma_{n_j}\subset K(f)$ and the domains $(D_n)$ satisfy condition~(b), there exists $N\in\N_0$ such that
\begin{equation}\label{kfc-cond1}
f^n(z_{n_j})\in D_{n_j+n},\quad\text{for } n>N.
\end{equation}
By property~(ii) and condition~(a),
\begin{equation}\label{bdry-cond1}
f(z_{n_j})\in \C\setminus D_{n_j+1},
\end{equation}
so the minimal integer~$N$ such that \eqref{kfc-cond1} holds is at least $1$. Define $n_{j+1}=n_j+N$, where~$N$ is this minimal integer. Then, by \eqref{kfc-cond1} and the minimality of $ N $,
\begin{equation}\label{kfc-cond2}
f^n(z_{n_j})\in D_{n_j+n},\quad\text{for } n>n_{j+1}-n_j,
\end{equation}
and
\[
f^{n_{j+1}-n_j}(z_{n_j})\in \C\setminus D_{n_{j+1}}.
\]
Moreover, $f^{n_{j+1}-n_j}(z_{n_j})\notin \partial D_{n_{j+1}}$, by condition~(a) and \eqref{kfc-cond2}, so
\begin{equation}\label{cond2}
f^{n_{j+1}-n_j}(z_{n_j})\in \C\setminus \overline{D}_{n_{j+1}}.
\end{equation}
Also, by property~(iii),
\begin{equation}\label{cond3}
f^{n_{j+1}-n_j}(z'_{n_j})\in D_{n_{j+1}}.
\end{equation}
It follows from~\eqref{cond2} and~\eqref{cond3} that the continuum $f^{n_{j+1}-n_j}(\Gamma_{n_j})$ includes points from both $D_{n_{j+1}}$ and $\C\setminus \overline{D}_{n_{j+1}}$ (see Figure \ref{LemFig}).

\begin{figure}[h]
    \centering
     \setlength\fboxsep{10pt}
     \setlength\fboxrule{0.5pt}
     \def\svgwidth{300pt}
     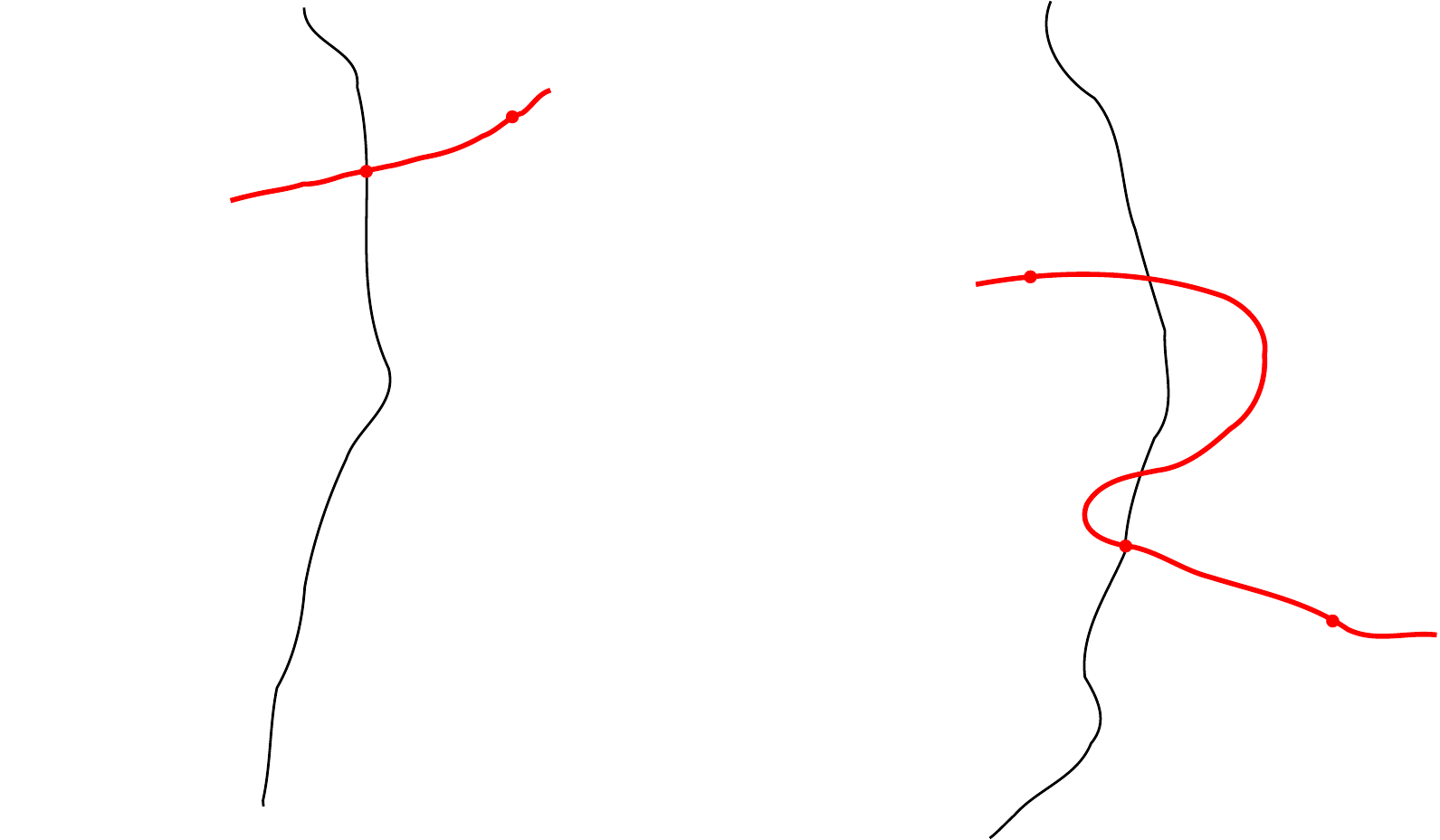
     \vspace{.2cm}
     \caption{Proof of Lemma~\ref{genctm}. }
      \label{LemFig}
\end{figure}

Now let $\Gamma_{n_{j+1}}$ be the component of the closed set
\[
f^{n_{j+1}-n_j}(\Gamma_{n_j})\cap (\C\setminus D_{n_{j+1}})
\]
that contains the point
\[
z'_{n_{j+1}}:=f^{n_{j+1}-n_j}(z_{n_j}).
\]
Then we deduce that $\Gamma_{n_{j+1}}$ meets $\partial D_{n_{j+1}}$ by applying Lemma~\ref{Newman} with
\[
E_0=f^{n_{j+1}-n_j}(\Gamma_{n_j})\cap (\C\setminus D_{n_{j+1}}) \quad \text{and} \quad E_1 = E_0\cap \partial D_{n_{j+1}}.
\]

Thus there exists $z_{n_{j+1}}\in \Gamma_{n_{j+1}}\cap\partial D_{n_{j+1}}$. Therefore, properties~(i) and~(ii) hold with $n_j$ replaced by $n_{j+1}$, and property~(iii) also holds, since
\[
f^n(z'_{n_{j+1}})=f^{n+n_{j+1}-n_j}(z_{n_j})\in D_{n_{j+1}+n},\quad\text{for } n\in\N,
\]
by \eqref{kfc-cond2}.
\end{proof}

\begin{remark*}
Note that property~(i) in Lemma~\ref{genctm} could be weakened to the property
\[
\Gamma_{n_j} \subset \{z: f^n(z) \in D_{n_j + n} \mbox{ for } n>N=N(z)\}\cap (\C \setminus D_{n_j}),
\]
since the only place in the proof where we use the fact that $\Gamma_{n_j} \subset K(f)$ is to deduce~\eqref{kfc-cond1} and it is clear that, if $\Gamma_{n_j}$ satisfies this weaker property, then any point $z \in \Gamma_{n_{j+1}} \subset f^{n_{j+1} - n_j}(\Gamma_{n_j})$ must satisfy property~(iii) with $n_j$ replaced by~$n_{j+1}$.
\end{remark*}

Next, we state two further topological lemmas that are needed for the proof of Theorem \ref{fullresult}. The first is a useful characterisation of a disconnected subset of the plane.

\begin{lemma}\cite[Lemma 3.1]{R11}
\label{rempe}
A subset $ S $ of $ \C $ is disconnected if and only if there exists a closed, connected set $ E \subset \C $ such that $ S \cap E = \emptyset $ and at least two different components of $ E^c $ intersect $ S. $
\end{lemma}

We also need the following generalisation of \cite[Lemma 1]{RS09}, given in \cite{Sixsmithmax}.  This result will be used again later in the paper.
\begin{lemma}
\label{RSlemm}
Let $(E_j)_{j \in \N_0}$ be a sequence of compact sets in $ \C $, $(m_j)_{j \in \N_0}$ be a sequence of positive integers and $f$ be a {\tef} such that $E_{j+1} \subset f^{m_j}(E_j )$, for $j \in \N_0$. Set $p_k = \sum_{j=0}^k m_j$, for $k \in \N_0$. Then there exists $\zeta\in E_0$ such that
\begin{equation*}
\label{feq}
f^{p_k}(\zeta) \in E_{k+1}, \qfor k \in \N_0.
\end{equation*}
\end{lemma}

We now give the proof of Theorem \ref{fullresult}.

\begin{proof}[Proof of Theorem \ref{fullresult}]
Suppose that $ \kfc $ is disconnected. Then, by Lemma~\ref{rempe}, there is a closed, connected set $ E \subset K(f) $ such that two distinct components of $ E^c $, say $ G_1 $ and $ G_2 $, each meet $ \kfc $.  Evidently the boundaries of $ G_1 $ and $ G_2 $ are connected and are contained in $ K(f). $ By Theorem~\ref{bounddomkfc}(a), we deduce that $G_1$ and $G_2$ are both unbounded, so $\partial G_1$ and $\partial G_2$ are unbounded, as are their images under the iterates of~$f$, by conditions~(a) and~(b).

We now show that there exists $n_0\in \N$ and a continuum $ \Gamma_{n_0} $ that satisfies properties (i),~(ii) and~(iii) in Lemma~\ref{genctm},  with $ j = 0 $.

Without loss of generality we can assume, by condition~(b), that for some point  $\alpha \in \partial G_1$  every domain $D_n$, $n\in\N_0$, contains the entire orbit  of~$\alpha$.  By Lemma~\ref{Newman}, there exists a continuum $\Gamma \subset \partial G_1$ such that  $\alpha \in \Gamma$  and $\Gamma \cap \partial D_0 \ne \emptyset$.

Now take $z_0\in \Gamma \cap \partial D_0$ and choose $N\in \N$ such that $ f^n(z_0)\in D_N $, for $ n\in \N_0. $  Since $f(z_0) \in \C \setminus D_1$, by condition~(a), it follows that the maximal value $n_0$ of~$n$ such that $f^n(z_0) \in \C \setminus D_n$ satisfies $1\le n_0 <N$. Note that  $ z'_{n_0} := f^{n_0}(z_0) $ lies outside $ \overline{D}_{n_0} $, so $ f^{n_0}(\Gamma)\cap \partial D_{n_0} \ne \emptyset$.

 Now let $\Gamma_{n_0}$ be  the component of $f^{n_0}(\Gamma)\setminus D_{n_0}$ that contains  $z'_{n_0}$. Then $\Gamma_{n_0}$ meets $\partial D_{n_0}$, by Lemma~\ref{Newman} again. It follows that the continuum $\Gamma_{n_0}$ satisfies

\begin{enumerate}[(i)]
\item $ \Gamma_{n_0} \subset K(f) \cap (\C\setminus D_{n_0})$;
\item there is a point $ z_{n_0} \in \Gamma_{n_0} \cap \partial D_{n_0}$;
\item there is a point $ z'_{n_0} \in \Gamma_{n_0} $ such that $ f^n(z'_{n_0}) \in D_{n_0+n} $ for all $ n \in \N$.

\end{enumerate}
Thus by Lemma~\ref{genctm} there is   a strictly increasing sequence $ (n_j)_{j \in \N_0} $ and a sequence of continua $(\Gamma_{n_j})_{j\in \N_0}$ such that, for each $j\in \N_0$,
\begin{enumerate}[(i)]
\item $ \Gamma_{n_j} \subset K(f) \cap (\C\setminus D_{n_j})$;
\item there is a point $ z_{n_j} \in \Gamma_{n_j} \cap \partial D_{n_j}; $
\item there is a point $ z'_{n_j} \in \Gamma_{n_j} $ such that  $ f^n(z'_{n_j}) \in D_{n_j+n} $   for all $ n \in \N; $
\item $f^{n_{j+1}-n_j}(\Gamma_{n_j})\supset \Gamma_{n_{j+1}}$.
\end{enumerate}
We now apply Lemma~\ref{RSlemm} with
\[
E_j=\Gamma_{n_j}\quad \text{and} \quad m_j=n_{j+1}-n_j,\qfor j\in\N_0.
\]
Then, by property~(iv),
\[
f^{m_j}(E_j)\supset E_{j+1},\qfor j\in \N_0.
\]
We deduce that there exists $\zeta\in E_0=\Gamma_{n_0}$ such that
\[
f^{p_k}(\zeta)\in E_{k+1}, \qfor k\in \N_0,
\]
where $p_k= m_0+\cdots +m_k=n_{k+1}-n_0$; that is,
\[
f^{n_{k+1}-n_0}(\zeta)\in \Gamma_{n_{k+1}}, \qfor k\in \N_0.
\]
Thus by   condition~(b) in the statement of the theorem   and property~(i) of the sequence of continua $(\Gamma_{n_j})$,
\[
f^{n_{k+1}-n_0}(\zeta)\to\infty \;\text{ as } k\to \infty,
\]
so $\zeta\in \kfc$, which contradicts the fact that $\zeta \in \Gamma_{n_0}\subset K(f)$. This completes the proof of Theorem~\ref{fullresult}.
\end{proof}
\begin{remark*}
This proof shows that, under the hypotheses of Theorem~\ref{fullresult}, if $K$ is any closed connected set in $K(f)$ and  $\alpha \in K$, then there is a positive constant  $C(K,\alpha)$  such that  $K\subset \{z:|z|\le C(K,\alpha)\}$.    It follows that, under the hypotheses of Theorem~\ref{fullresult}, any Fatou component of~$f$ contained in $K(f)$ is bounded.

If~$f$ is \spl, then this conclusion about Fatou components is already known and moreover such Fatou components cannot be wandering domains \cite[Theorem~1.4]{O12a}. However, under this hypothesis~$K(f)$ may contain an unbounded closed connected set; see Example~5.2.
\end{remark*}

\section{Components of $ \kfc $}
\label{components}
\setcounter{equation}{0}

In this section we prove Theorem \ref{uncount} and a number of other results about the components of $\kfc$. 

We begin by proving the following,   which parallels the result for $ I(f) $ in \cite[Theorem 1.2]{RS14}.    Several of the results in this section follow from this result with appropriate choices of the set $ E $.

\begin{theorem}\label{kfccomp}
Let $f$ be a {\tef} and let $E$ be a set such that $E \subset \kfc$ and $J(f) \subset \overline{E}$. Then either $\kfc$ is connected or it has infinitely many components that meet $E$.
\end{theorem}

 The proof of Theorem \ref{kfccomp} uses  the well known \emph{blowing up property} of the Julia set,   stated as the next lemma.   Here $ E(f) $ is the \textit{exceptional set} of $ f $, that is, the set of points with a finite backwards orbit under $ f $ (which for a \tef\ contains at most one point).

\begin{lemma}
\label{blowup}
Let $ f $ be an entire function, let $ K $ be a compact set such that $ K \subset \C \setminus E(f) $ and let $ G $ be an open \nhd\ of $ z \in J(f) $. Then there exists $ N \in \N $ such that $ f^n(G) \supset K $, for all $ n \geq N. $
\end{lemma}

 The proof of Theorem \ref{kfccomp} also uses the following result.   This was proved in the special case that $F = I(f)$ in \cite[Theorem 5.1(a)]{RS11}.

\begin{lemma}\label{EF}
Let~$f$ be a {\tef}, and let~$E$ and $F$ be sets such that $E\subset F$, $F$ is backwards invariant, and $J(f) \subset \overline{E}$. If~$E$ meets only finitely many components of~$F$, then $F \cap J(f)$ lies in a single component of $F$.
\end{lemma}
\begin{proof}
Suppose that $E$ is contained in the union of finitely many
components of $F$, say $F_1,F_2,\ldots,F_m$. Take any $z\in
F\cap J(f)$. Since $J(f)\subset\overline{E}$, there exist
$z_n\in E$ such that $z_n\to z$ as $n\to\infty$. Without loss of
generality all terms of this sequence $(z_n)$ lie in a single
component, $F_j$ say. Since $z\in F$, we have $z\in F_j$.
Hence
\begin{equation}\label{eqn5.1}
F\cap J(f)\subset F_1\cup F_2\cup \cdots\cup F_m.
\end{equation}
We now assume that $F_1,F_2,\ldots,F_m$ is the minimal set of
components of $F$ such that (\ref{eqn5.1}) holds. Then $F_j\cap
J(f)\ne\emptyset$, for $j=1,2,\ldots,m$. Note that if the exceptional set $E(f)$ is non-empty, then
\begin{equation}\label{eqn5.1a}
(F_j\setminus E(f))\cap J(f)\ne \emptyset,\quad\text{for } j=1,2,\ldots, m.
\end{equation}
Indeed, if $E(f)=\{\alpha\}\subset F_j\cap J(f)$, then it follows from Lemma~\ref{blowup} that $\alpha$ is a limit point of the backwards orbit of
any non-exceptional point in $F\cap J(f)$ and hence $\alpha$ is the limit of a sequence in $F_i\cap J(f)$, say, by (\ref{eqn5.1}). Thus $i=j$ and so (\ref{eqn5.1a}) holds.

If $m=1$, then $F\cap
J(f)$ is contained in one component of $F$, as required.
If $m>1$, then we can take $z_1\in F_1\cap J(f)$ and an open
disc $D$ centred at~$z_1$ so small that
\begin{equation}\label{eq5.2}
D\cap(F_2\cup\cdots \cup F_m)=\emptyset.
\end{equation}
Consider $F_j$, $j\ge 2$. Then
there exists $N\in\N$ such that $f^N(D)$ meets both $F_1\cap
J(f)$ and $F_j\cap J(f)$, by (\ref{eqn5.1a}) and Lemma~\ref{blowup}. Hence there
exist $w_1,w_j\in D$ such that
\[
f^N(w_1)\in F_1\cap J(f)\quad\text{and}\quad f^N(w_j)\in F_j\cap J(f),
\]
so $w_1,w_j\in F_1$ by the backwards invariance of $F\cap
J(f)$ and (\ref{eq5.2}). Thus $f^N(F_1)$ is a connected subset of
$F$ that meets both $F_1$ and $F_j$, which is a contradiction. This completes the proof.
\end{proof}

  We now deduce Theorem~\ref{kfccomp} from Lemma~\ref{EF}.

\begin{proof}[Proof of Theorem~\ref{kfccomp}]
By Lemma~\ref{EF} with $F = \kfc$ we deduce that, in order to prove Theorem~\ref{kfccomp}, it is sufficient to show that, if $\kfc \cap J(f)$ lies in a single component of $\kfc$, say $C_1$, then $\kfc$ must be connected.

It follows from Lemma~\ref{kfcjf} that, if there is a component of $\kfc$ that does not meet $J(f)$, then it must be a Baker domain with a disconnected boundary, or a preimage of such a Baker domain. Suppose then that~$U$ is such a component. Then $U$ has more than one complementary component, each of which is closed and unbounded, and meets $J(f)$.

All the points of $\kfc\cap J(f)$ lie in these complementary components of $U$, and $\kfc\cap J(f)$ cannot be contained in a single complementary component of~$U$ because $J(f)=\overline{\kfc\cap J(f)}$, by Theorem~\ref{Qprops}. Hence the component~$C_1$ of $\kfc$   that   contains $\kfc \cap J(f)$ must meet at least two complementary components of $U$ and so it must meet the boundaries of these two complementary components, which are subsets of $\partial U$. This contradicts the fact that $\partial U\cap \kfc=\emptyset$. Hence such a component $U$ of $\kfc$ cannot exist. Thus any component of $\kfc$ must meet $J(f)$ and hence must lie in $C_1$; that is, $\kfc$ is connected. This completes the proof.
\end{proof}

We now show that several connectedness properties of $\kfc$ follow easily from Theorem~\ref{kfccomp}. First,   noting Eremenko's result~\cite{E} that $J(f) = \partial I(f)$, we apply Theorem~\ref{kfccomp} with $E = I(f)$ to give the following.

\begin{corollary}\label{IK}
Let $f$ be a {\tef}. If $I(f)$ is connected, then $\kfc$ is connected.
\end{corollary}

Next, we give  conditions for $\kfc$ and $ \kfc \cap J(f) $ to be spiders' webs.   We say that a connected set $ E $ is a \emph{spider's web} if there exists a sequence $ (G_n)_{n \in \N} $ of bounded, \sconn\ domains such that
\begin{equation*}
\label{web}
G_n \subset G_{n+1}  \text{ and } \partial G_n \subset E, \text{ for each } n \in \N, \; \text{ and } \bigcup_{n \in \N} G_n = \C.
\end{equation*}
Clearly, any connected set that contains a \sw\ is itself a \sw,  so it follows from Corollary~\ref{IK} that if $ I(f) $ is a \sw, then $ \kfc $ is a \sw.  In fact, we prove the following more general result.

\newpage

\begin{corollary}\label{SW}
Let $f$ be a {\tef}.
\begin{enumerate}[(a)]
\item If $\kfc$ contains a \sw, then $\kfc$ is a \sw.
\item If $ \kfc \cap J(f) $ contains a \sw, then $\kfc \cap J(f) $ is a \sw.
\end{enumerate}
\end{corollary}

We prove Corollary \ref{SW} by using various properties of the subset of $I(f)$ known as the fast escaping set $A(f)$, which was introduced in~\cite{BH99}. In particular, we use the facts that $J(f) = \partial A(f)$, that all components of $A(f)$ are unbounded, and that all components of $ A(f) \cap J(f) $ are unbounded whenever $ f $ has no \mconn\ Fatou components.  For proofs of these properties we refer to~\cite{RS10a}, for example.

\begin{proof}[Proof of Corollary \ref{SW}]
To prove part (a),  suppose that $\kfc$ contains a \sw\ and let $C_1$ be the component of $\kfc$ containing the \sw. Then, since all the components of $A(f)$ are unbounded, we have $A(f) \subset C_1$. Also, since $J(f) = \partial A(f)$, we can apply Theorem~\ref{kfccomp} with $E = A(f)$. It follows that $\kfc$ is connected and hence   that   $\kfc$ is a \sw.

 The proof of part (b) is similar.  If $\kfc \cap J(f) $ contains a \sw, then $ f $ has no \mconn\ Fatou components by \cite[Theorem 3.1]{iB84}.  Hence every component of $ A(f) \cap J(f) $ is unbounded, so if $C_2$ is the component of $\kfc \cap J(f)$ that contains the \sw, we have $A(f) \cap J(f) \subset C_2$.  The result then follows by applying Theorem~\ref{kfccomp} with $E = A(f) \cap J(f)$.
\end{proof}

Finally in this section, we prove Theorem \ref{uncount}.  In fact, we prove the following slightly stronger result.

\begin{theorem}
\label{uncount2}
Let $ f $ be a \tef\ such that $ \kfc $ is disconnected.  Then:
\begin{enumerate}[(a)]
\item $ \kfc $ has infinitely many unbounded components;
\item every \nhd\ of a point in $ J(f) $ meets uncountably many components of $ \kfc. $
\end{enumerate}
\end{theorem}

The proof of Theorem \ref{uncount2}(b) is closely related to other proofs of this type (for example, \cite[Theorem 1.3]{O12a}), but we give it in full for the convenience of the reader.

\begin{proof}[Proof of Theorem \ref{uncount2}]
For part~(a), we again apply Theorem~\ref{kfccomp} with $E = A(f)$ and conclude that, if $\kfc$ is disconnected, then there are infinitely many components of $\kfc$ that meet~$A(f)$. Since, as noted earlier, all components of $A(f)$ are unbounded, the result follows.

To prove part~(b), note first that, since $ \kfc $ is disconnected, it follows from Lemma~\ref{rempe} that there exists a closed connected set $ \Gamma \subset K(f) $ with two complementary components, say $ G_0 $ and $ G_1 $, each containing points in $ \kfc $.

Now since the boundaries of $ G_0 $ and $ G_1 $ are connected and lie in $ K(f), $ both $ G_0 $ and $ G_1 $ must meet $ J(f). $  For suppose that $ G_0 $, say, lies in $ F(f) $.  Then since $ \partial G_0 \subset K(f),$ it follows that $ G_0 $ is a Fatou component that is also a component of $ \kfc. $  However, since $ \partial G_0 $ is connected, Lemma \ref{kfcjf} shows that this is impossible. So, for $i=0,1$, there exist $z_i \in J(f)$ and a bounded open neighbourhood $H_i$ of~$z_i$ such that $\overline{H_i} \subset G_i \setminus E(f)$.

Also, since $J(f)$ is unbounded, for each $ n \geq 2 $ there is a point $ z_n \in J(f) $ and a bounded open neighbourhood $H_n$ of $ z_n, $ with the properties that $H_n \cap E(f) = \emptyset$ and $ \inf_{z \in H_n} |z| \to \infty $ as $ n \to \infty $.

Now let $ z $ be an arbitrary point in $ J(f) $ and let $ V $ be a bounded open \nhd\ of $ z $.  Then, by Lemma \ref{blowup}, there exists $ k \in \N $ such that
\begin{equation}
\label{blowup1}
f^k(V) \supset \overline{H}_0 \cup \overline{H}_1,
\end{equation}
and, for any $ n \geq 2 $, there exists $ m_n \in \N $ such that
\begin{equation}
\label{blowup2}
f^{m_n}(H_0) \supset \overline{H}_n, \; f^{m_n}(H_1) \supset \overline{H}_n \mbox{ and } f^{m_n}(H_n) \supset \overline{H}_0 \cup \overline{H}_1.
\end{equation}

Now let $ s = s_1s_2s_3 \ldots $ be an infinite sequence of $ 0 $s and $ 1 $s.  We will show that there is an uncountable set of such sequences that encode the orbits of points that lie in distinct components of $ \kfc $ that meet $ \overline{V}$.

To show this, put $ E_0 = \overline{V} $ and, for $ n \in \N, $ set $$ E_{2n} = \overline{H}_{n+1} \textrm{  and  } E_{2n-1} = \overline{H}_{s_n}. $$

Then, for each sequence $ s = s_1s_2s_3 \ldots $, it follows from (\ref{blowup1}), (\ref{blowup2}) and Lemma~\ref{RSlemm} that there is a  corresponding sequence $ (p_n)_{n \in \N} $ and a point $ \zeta_s \in \overline{V} $ such that $ f^{p_n}(\zeta_s) \in E_n $ for $ n \in \N.$ Furthermore, all such points must lie in $ \kfc. $

We now claim that points in $ \overline{V} \cap \kfc $ whose orbits are encoded by different infinite sequences of $ 0 $s and $ 1 $s must lie in different components of $ \kfc $.  For if two such sequences differ, then some iterate of $ f $ will map one point to $ G_0 $ and the other to $ G_1 $.  Thus, if the two points are in the same component $ C $ of $ \kfc $, then some iterate of $ C $ meets $ \Gamma \subset K(f), $ which is a contradiction.

Evidently, the set of all sequences $ s = s_1s_2s_3 \ldots $ of $ 0 $s and $ 1 $s can be put in one-to-one correspondence with the binary representations of points in the unit interval. We have therefore shown that every \nhd\ of an arbitrary point in $ J(f) $ meets uncountably many components of $ \kfc $, and this proves part~(b).
\end{proof}

\begin{remark*}
It follows by a similar argument to the proof of Theorem \ref{uncount2}(b) that, for a \tef\ $ f $, every \nhd\ of a point in $ J(f) $ meets uncountably many components of $ \kfc \cap J(f). $  The proof uses Lemma~\ref{RSlemm} with the compact sets $E_n$, $n\ge 0$, in the above proof replaced by $E_n\cap J(f)$, $n\ge 0$.
\end{remark*}

\section{Examples}
\label{examples}
\setcounter{equation}{0}

In this section, we give details of the two examples referred to in the remarks after the statement of Theorem \ref{fullresult}.

First, we give an example of a \tef\ that satisfies the conditions of Theorem \ref{fullresult} but not those of Theorem \ref{minconn}.

\begin{example}
\label{genex}
Let $ f $ be the \tef\ defined by
\[ f(z) = - 10 z e^{-z} - \tfrac{1}{2} z. \]
Then $ f $ satisfies the conditions of Theorem \ref{fullresult}, so $ \kfc $ is connected, but there is no $ r>0 $ such that $ m^n(r) \to \infty $ as $ n \to \infty. $
\end{example}

\begin{proof}  Since
\[
m(r)\le |f(r)|=\tfrac12 r(1+o(1))\;\text{ as } r\to \infty,
\]
it is clear that there is no $ r>0 $ such that $ m^n(r) \to \infty $ as $ n \to \infty. $

We show, nevertheless, that $ \kfc $ is connected because $ f $ satisfies the conditions of Theorem \ref{fullresult}.

Let $ (D_n)_{n \in \N} $  be the sequence of nested domains defined by
\[
\begin{split}
D_n = & \left\{ z \in \C: 0 < \textrm{Re } z < 4n \pi, \, \, \left| \textrm{Im } z \right| < 4n \pi \right\} \\
& \quad \cup \, \left\{ z \in \C: - n \pi < \textrm{Re } z \leq 0, \, \, \left| \textrm{Im } z \right| < n \pi \right\}.
\end{split}
\]
(Here, for convenience, we have labelled the domains with subscripts in $ \N $ rather than in $ \N_0 $.)  Each domain $ D_n $ is the union of two rectangles, the larger in the right half-plane and the smaller in the left half-plane (see Figure \ref{FigEx}).  It is clear that condition~(b) in the statement of Theorem \ref{fullresult} is satisfied by this sequence of domains.

We now show that condition~(a) in the statement of Theorem \ref{fullresult} is also satisfied for $ n > 1 $.  To see this, consider Figure \ref{FigEx}, in which sections of the boundary of~$ D_n $ are labelled with lower case letters and their images under $ f $  with the corresponding upper case letters.  The following brief notes discuss the images of different sections of the boundary of $ D_n $ using the same labelling as in the figure.

\underline{Section a} \, On this section $ z = 4n\pi + iy $, where $ - 4n\pi \leq y \leq 4n\pi, $ so
\[
|f(z)-(-\tfrac12 z)|=|f(z)-(-2n\pi-\tfrac12 i y)|\le 40\pi\sqrt 2 \exp(-4\pi)<10^{-3},
\]
and hence $f(z)$ evidently lies outside $ D_{n+1} $ for $ n > 1. $

\underline{Section b} \, On this section $ z = x + 4n\pi i $, where $ 0 \leq x \leq 4n\pi $, so
\[
f(z)=-\tfrac12 x-10xe^{-x} +i\left(-2n\pi - 40n\pi e^{-x}\right),
\]
lies in the left half-plane, below the line $\textrm{Im } z= -2n\pi$, and hence outside $ D_{n+1} $. The image of this section meets the imaginary axis at $ f(4n\pi i) = -42 n \pi i $  (off the scale in Figure \ref{FigEx}).

\begin{figure}[h]
     \centering
     \setlength\fboxsep{10pt}
     \setlength\fboxrule{0.5pt}
     \def\svgwidth{300pt}
     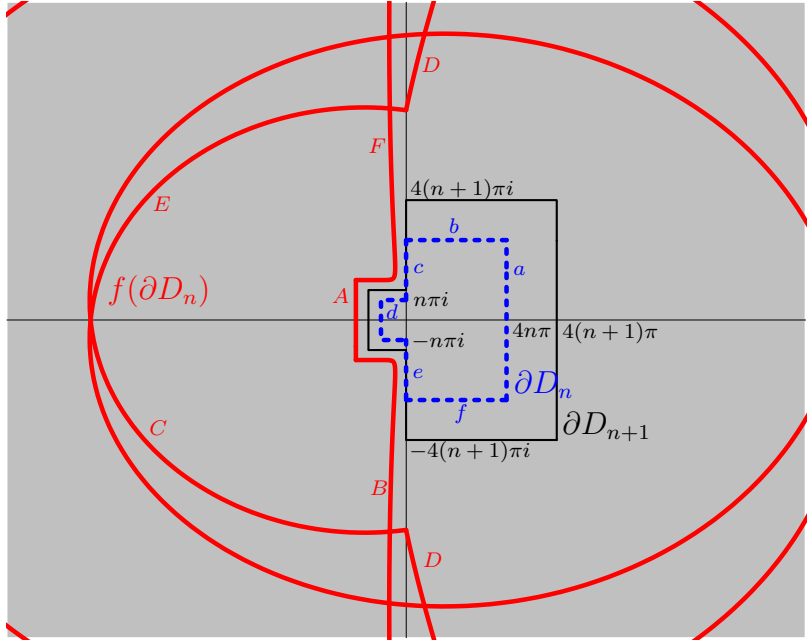
     \vspace{.2cm}
     \caption{The boundary of the domain $ D_n $ and its image under the function $ f: z \mapsto - 10 z e^{-z} - \tfrac{1}{2} z $ (see Example \ref{genex}), illustrated for $ n = 2 $.  Sections of the boundary of $ D_n $ are labelled with lower case letters and their images with the corresponding upper case letters.  Only the immediate vicinity of $ D_n $ is shown.}
      \label{FigEx}
\end{figure}

\underline{Section c} \, Here $ z = iy $, where $ n \pi \leq y \leq 4n\pi. $  The term $ -10ze^{-z} $ is dominant, and since this section has length $ 3n \pi $, the image of the section winds around the origin, the factor of $ 10 $ ensuring that it stays outside $ D_{n+1}. $  The image is thus a spiralling curve joining $ f(4n\pi i)= -42\pi i $ to  $ f(n\pi i) = (19/2) n \pi i $.

\underline{Section d} \, This section of the boundary is the union of the three line segments
\[
\begin{split}
z = & \, x + n \pi i, \quad - n\pi \leq x \leq 0, \\
z = & - n \pi +  iy, \quad - n\pi \leq y \leq n \pi, \\
z = & \, x - n \pi i, \quad - n\pi \leq x \leq 0. \\
\end{split}
\]
On all three segments the modulus of $ 10ze^{-z} $ is at least $10n\pi$ and exceeds the modulus of $\tfrac12 z$ by a factor of at least 20, so the images of these segments lie well outside $D_{n+1}$. Most of these images lies off the scale in Figure \ref{FigEx}.

\underline{Sections e and f} \,  The images of these sections of the boundary are the reflections in the real axis of the images of sections c and b, respectively.

We have now shown that the whole of $ f (\partial D_n) $ lies outside $ D_{n+1} $, so the conditions of Theorem \ref{fullresult} are satisfied and  hence  $ \kfc $ is connected.
\end{proof}

Next, we show that not all \tef s that are \spl\ in the sense defined in \cite{O12a} meet the conditions of Theorem \ref{fullresult}. The authors are grateful to Dave Sixsmith for suggesting this example.

\begin{example}
\label{counter}
Let $ f $ be the \tef\ defined by
\[ f(z) = \cos z + z. \]
Then $ f $ is \spl\ and $ \kfc $ is disconnected.
\end{example}

\begin{proof}
By \cite[Theorem 1.6]{O12a}, a \tef\ $ f $ is strongly poly\---\ nomial-like if and only if there exists a sequence of bounded, \sconn\ domains $ (D_n)_{n \in \N_0} $ such that
\begin{enumerate}[(i)]
\item $ f(\partial D_n) $ surrounds $ \overline{D}_n $, for $ n \in \N_0, $
\item $ \bigcup_{n \in \N_0} D_n = \C $, and
\item $ \overline{D}_n \subset D_{n+1} $, for $ n \in \N_0 $.
\end{enumerate}
Let $ (D_n)_{n \in \N_0} $  be the sequence of nested open rectangles defined by
\[D_n = \left\{ z \in \C: -\left( 2n+\dfrac{11}{4} \right) \pi < \textrm{Re } z < \left( 2n+\dfrac{9}{4} \right) \pi, \, \, \left| \textrm{Im } z \right| < 2(n+1) \pi \right\}.\]
Then properties (ii) and (iii) are evidently satisfied.

To show that property (i) is also satisfied, consider first the images under $ f $ of the  two  sides of the rectangle $ D_n $ parallel to the real axis.  Writing $$ f(z) = \tfrac{1}{2}(e^{iz} + e^{-iz}) + z, $$ we see that $ f $ maps both  these  sides into the annulus
\[
\{z:\tfrac12 e^{2(n+1)\pi}-4(n+1)\pi<|z|< \tfrac12 e^{2(n+1)\pi}+4(n+1)\pi\},
\]
which clearly lies outside $ \overline{D}_n $ for all $ n \in \N_0. $

Next, if $ z $ lies on the vertical side of $ D_n $ in the right half-plane, we have
\[ \textrm{Re } f(z) = \textrm{Re } z + \dfrac{1}{2 \sqrt{2}}(e^y + e^{-y}), \]
where $ y = \textrm{Im } z$.  Similarly, for points on the vertical side of $ D_n $ in the left half-plane,
\[ \textrm{Re } f(z) = \textrm{Re } z - \dfrac{1}{2 \sqrt{2}}(e^y + e^{-y}).\]
It follows that $ f(\partial D_n) $ surrounds $ \overline{D}_n $, for $ n \in \N_0, $ so $ f $ is \spl.

 Now, $ f $  has fixed points at $ z = (k + \tfrac{1}{2})\pi,\, k \in \Z $.  These fixed points are repelling if $ k $ is odd and superattracting if $ k $ is even, and all points on the real axis except for the repelling fixed points tend under iteration towards the nearest superattracting fixed point.  It follows that the real axis lies in $ K(f) $, and indeed that the real axis is a closed, connected set in $ K(f) $ that disconnects $ \kfc $.
\end{proof}

\begin{remark*}
It can be shown that $ K(f) $ is connected for the function $ f $ in Example~\ref{counter}.  We omit the details.
\end{remark*}

\end{document}